\documentclass{birkjour}
 \newtheorem{theorem}{Theorem}[section]

 \newtheorem{lemma}[theorem]{Lemma}
 \newtheorem{proposition}[theorem]{Proposition}
 \theoremstyle{definition}
 
 \theoremstyle{remark}
 \newtheorem{remark}[theorem]{Remark}
 
 \numberwithin{equation}{section}
\usepackage{hyperref}

\begin{document}

%
%
%
%
%
%
%
%
%

\title[Affine minimal translation surface]
 {On the height function of the affine minimal translation surfaces}

\author[M.S. Lone]{Mohamd Saleem Lone}

\address{%
International Centre for Theoretical Sciences, \\
Tata Institute of Fundamental Research,\\
 560089, Bengaluru, India.}

\email{saleemraja2008@gmail.com, \\mohamdsaleem.lone@icts.res.in}




\date{March 30, 2017}

\begin{abstract}
In this paper, using the Weierstrass-Enneper formula and the hodographic coordinate system, we find the relationships between the Ramanujan identity and the generalized class of Scherk surfaces known as affine Scherk surfaces. We find the Dirichlet series expansion of the affine Scherk surface. We also obtain some of the probability measures of affine Scherk surface with respect to its logarithmic distribution.  Next, we classify the affine minimal translation surfaces in $\mathbb{L}^3$ and remarked the analogous forms in $\mathbb{L}^3.$
\end{abstract}

\maketitle
\section{Introduction}
In recent times, various relations between the minimal surfaces and the Ramamujan identities have been explored. The first such relations were found by Kamein \cite{kamien2001decomposition}, wherein using Ramanujan's identities, Kamein obtained a decomposition of the height function of the Scherk's first surface. In topological defects (particularly a screw dislocation) in a quiescent layered structure, Scherk's first surface acts an infinite superposition of these defects (see \cite{kamien2001decomposition} for more applications).  In \cite{dey2016ramanujan}, Dey found some non-trivial identities using Ramanujan's identities and the Weierstrass-Enneper (W-E) representation of minimal surfaces, and the analogue for Born-Infeld solitons.

Let $z=z(x,y)$ be a local form of a surface, the minimal surface condition, i.e., mean curvature $H \equiv 0$ is equivalent to the second-order differential equation:
\begin{equation}\label{a1.1}
(1+z_y^2)z_{xx}-2z_xz_yz_xy{}+(1+z_x^2)z_{yy}=0.
\end{equation} 
The W-E representation gives a general parametric form solution of a minimal surface in the neighbourhood of a non-umbilical interior point in the following form:
\begin{eqnarray}\label{x1.1}
\left\{
\begin{array}{ll}\vspace{.1cm}
x(\xi)=x_0+ \Re \int_{\xi_0}^\xi R(\omega)(1-\omega^2)d\omega\\\vspace{.1cm}
y(\xi)=y_0+\Re \int_{\xi_0}^\xi \iota R(\omega)(1+\omega^2)d\omega\\ 
z(\xi)=z_0+\Re \int_{\xi_0}^\xi 2R(\omega)d\omega,
\end{array}
\right.
\end{eqnarray}  
where $\xi$ is a complex parameter and $R(\omega)$ is a meromorphic function. In this parameterization $(\xi_1,\xi_2)$ acts as isothermal parameterization, where $\xi=\xi_1 + \iota \xi_2$. Using the hodographic coordinates, Dey found a much simpler first order analogous form of (\ref{a1.1}):
\begin{equation}\label{v1.3}\xi^2 \bar{z}_\xi+z_\xi=0.\end{equation}
Using (\ref{v1.3}), Dey rederived the W-E representation formula of minimal surfaces \cite{dey2003weierstrass}. For similar studies, see \cite{dey2013one, dey2018euler, dey2017born}.

In 3-dimensional Euclidean space ($\mathbb{E}^3$), a surface $\varphi(x,y)$ is called a translation surface if $\varphi(x,y)$ can be written as
\begin{equation}\label{p1.4}\varphi=f(x)+g(y),\end{equation}
where $f$ and $g$ are two planar (and/or non-planar) smooth curves lying in orthogonal (and/or non-orthogonal) planes, such that $f^\prime \times g^\prime \neq 0$. These two curves ($f$ and $g$) are called as the generators of the surface. Depending upon the location and nature of generators, lots of discussions have been done with respect to such minimal translation surfaces.  

Recently, L{\'o}pez and Perdomo obtained a classification of the long-standing problem of minimal translation surfaces where both the curves are non-planar \cite{lopez2017minimal}.

In case of planar curves where the generating curves lie in the orthogonal planes, Scherk proved that, apart from planes, the only minimal translation surfaces are the surfaces given by $$\varphi(x,y)=\log \frac{\cos y}{\cos x}.$$  Now let $\varphi=\varphi(x,y)$ be the local form of a surface in Lorentz-Minkowski space ($\mathbb{L}^3$) with the metric $g=dx_1^2+dx_2^2-dx_3^2$. The mean curvature formula is given by
\begin{equation}
H=\epsilon\frac{EN-2FM+GL}{2(EG-F^2)},
\end{equation}
where $\epsilon=1$ for a spacelike surface and $\epsilon=-1$ for a timelike surface.
The $H \equiv 0$ is equivalent to the following differential equation
\begin{equation}
(1-z_y^2)z_{xx}+2z_xz_yz_{xy}+(1-z_x^2)z_{yy}=0,
\end{equation} 
satisfying $z_x^2 +z_y^2<1.$
Similar to the minimal surface, the W-E representation for a maximal surface is given by \cite{dey2017born}
\begin{eqnarray}\label{ax1.1}
\left\{
\begin{array}{ll}\vspace{.1cm}
x(\xi)=x_0+ \Re \int_{\xi_0}^\xi R(\omega)(1+\omega^2)d\omega\\\vspace{.1cm}
y(\xi)=y_0+\Re \int_{\xi_0}^\xi \iota R(\omega)(1-\omega^2)d\omega\\ 
z(\xi)=z_0+\Re \int_{\xi_0}^\xi 2R(\omega)d\omega,
\end{array}
\right.
\end{eqnarray}  
Kobayashi \cite{kobayashi1983maximal} studied the minimal translation surfaces in $\mathbb{L}^3$  and proved the following:
Every maximal spacelike translation surface in $\mathbb{L}^3$ is  either a spacelike plane or is given by
\begin{equation*}\varphi(x,y)=\log \frac{\cosh y}{\cosh x},
\end{equation*}
where $\tanh^2 x+\tanh^2 y < 1.$ This surface is called the surface of Scherk of the first kind.

Recently, Liu and Yu \cite{liu2013affine} introduced a new class of translation surfaces given as a graph of $\varphi(x,y)=f(x)+g(y+ax)$ for some non-zero real constant $a$. The generators lie in the non-orthogonal planes $x=0$ and $y+ax=0$, therefore acts as a natural generalization of (\ref{p1.4}). This type of surface is called as affine translation surface. In case of minimal affine translation surfaces, we have the following:
\begin{theorem}\label{thm1.1}\cite{liu2017affine, liu2013affine}
	Let $\varphi$ be a minimal affine translation surface. Then $\varphi$ is either a linear function or is given by
	\begin{equation}\label{affineSC}
	\varphi(x,y)=\log \frac{\cos (y+ax)}{\cos(\sqrt{1+a^2}x)}.
	\end{equation}
\end{theorem} 
The surface in (\ref{affineSC}) is called as generalized Scherk surface or an affine Scherk surface.

Locally $z(x,y)$ is a  Born-Infeld soliton if it satisfies an equation \cite{whitham2011linear}:
\begin{equation}
(1+z_y^2)z_{xx}-2z_xz_yz_xy{}+(z_x^2-1)z_{yy}=0.
\end{equation}
The minimal (maximal) surfaces and the Born-Infeld solitons are very closely related in the sense that a minimal surface is obtained by a wick rotation in the variable $y\leftrightarrow \iota y$ in Born-Infled soliton equation and vice-versa \cite{dey2013one}. In the same way, wick rotating $x$ by $\iota x$ in the Born-Infeld soliton equation, we get maximal surface equation and vice-versa \cite{dey2017born}.

We extend the study by decomposing the height function of an affine minimal translation surface with the help of Ramanujan identity, W-E representation formula and hodographic coordinates. We also find their Dirichlet series expansion and find the probability mass function with the help logarithmic distribution.
\section{Affine Scherk surface and Ramanujan identity in $\mathbb{E}^3$}\label{sec2}
\begin{lemma}\label{lemma1} $R(\omega)$ for the affine Scherk surface is given by
	\begin{equation}\label{1.1}
	R(\omega)= \frac{\sqrt{1+a^2}(1+\iota a)}{(1+a^2)+(1+\iota a)^2 \omega^2}+ \frac{1}{1-\omega^2},
	\end{equation}
	where $a\in \mathbb{R}.$
	\begin{proof} In order to find $R(\omega)$, we use the method of hodographic coordinates. Let $z=x+ \iota y$ and $\bar{z}=x- \iota y$ be the complex coordinates and  $u=\varphi_{\bar{z}}=\frac{\varphi_x +\iota \varphi_y}{2}$, $v=\varphi_{z}=\frac{\varphi_x -\iota \varphi_y}{2}$. For the affine Scherk surface, we have
		\begin{equation*}
		\varphi_x= \sqrt{1+a^2}\tan (\sqrt{1+a^2}\Re(x))-a\tan [a\Re(z)+\Im(z)],\text{ } \varphi_y =-\tan[a\Re(z)+\Im(z)].
		\end{equation*}
		This implies that
		\begin{eqnarray}\label{4.2}
		a\Re(z)+\Im(z)=\tan^{-1}\left(\iota(u-v)\right).
		\end{eqnarray}
		and
		\begin{eqnarray}\label{4.1}
		\Re(z)=\frac{1}{\sqrt{1+a^2}}\tan^{-1}\left[\frac{(u+v)-\iota a (v-u)}{\sqrt{1+a^2}}\right].
		\end{eqnarray}
		Let $\xi=\frac{\sqrt{1+4uv}-1}{2v}$ and $\bar{\xi}=\frac{\sqrt{1+4uv}-1}{2u}$ be two new variables, with the inverse transformation $u=\frac{\xi}{1-\xi \bar{\xi}}$ and $v=\frac{\bar{\xi}}{1-\xi \bar{\xi}}$. In terms of $\xi$ and $\bar{\xi}$, (\ref{4.2}) and (\ref{4.1}) reduces to
		\begin{eqnarray*}\label{1.6}
			a \Re(z)+ \Im(z)=\tan^{-1}(\iota \xi)-\tan^{-1}(\iota \bar{\xi}),
		\end{eqnarray*}
		\begin{eqnarray}\label{1.7}
		\Re(z)=\frac{1}{\sqrt{1+a^2}}\tan^{-1}\left[\frac{1}{\sqrt{1+a^2}}\left\{\frac{\xi(1+\iota a)}{1-\xi \bar{\xi}}+\frac{\bar{\xi}(1-\iota a)}{1-\xi \bar{\xi}}\right\}\right].
		\end{eqnarray}
		Let $\varsigma(\xi)=\frac{\xi(1+\iota a)}{\sqrt{1+a^2}}$, $\bar{\varsigma}(\bar{\xi})=\frac{\bar{\xi}(1-\iota a)}{\sqrt{1+a^2}}$,  eq. (\ref{1.7}) is equivalent to
		\begin{eqnarray*}
			\Re(z)=\frac{1}{\sqrt{1+a^2}}\left(\tan^{-1}\varsigma + \tan^{-1}\bar{\varsigma}\right).
		\end{eqnarray*}
		Thus, we can write
		\begin{equation*}
		\bar{z}=F(\xi)+G(\bar{\xi}),
		\end{equation*}
		where
		\begin{eqnarray*}
			F(\xi)&=&\frac{1+ \iota a}{\sqrt{1+a^2}}\tan^{-1}\xi - \iota  \tan^{-1}(\iota \xi),\\
			G(\bar{\xi})&=&\frac{1+ \iota a}{\sqrt{1+a^2}}\tan^{-1}\bar{\xi} + \iota \tan^{-1}(\iota \bar{\xi}).
		\end{eqnarray*}
		We know that $R(\omega)=F^\prime(\omega)$ (see \cite{dey2003weierstrass}). This proves our claim in eq. (\ref{1.1}). 
		\begin{remark}In the above lemma, we see that the points $\pm 1, \pm \iota \frac{\sqrt{1+a^2}}{1+\iota a}$ are the poles of $R(\omega)$, which are precisely the umbilical points of minimal surface.
		\end{remark}
	\end{proof}
\end{lemma}
\begin{proposition}Let $\varphi$ be an affine non-planar minimal translation surface in $\mathbb{E}^3$ and $\xi$ in a neighbourhood of $\mathbb{C}$ away from the umbilical points $(\xi \neq\pm 1, \pm \iota \frac{\sqrt{1+a^2}}{1+\iota a})$, we have the following identity
	\begin{eqnarray*}
		&\log \frac{\cos (y+ax)}{\cos(\sqrt{1+a^2}x)}=\prod_{k=1}^\infty \log \left(\frac{\left(k-\frac{1}{2}\right)\pi-[\Im(\lambda_3(\xi))+2\Re(\lambda_4(\xi))+a(\Im(\lambda_1(\xi))+\Re(\lambda_2(\xi)))]}{\left(k-\frac{1}{2}\right)\pi-\sqrt{1+a^2}(\Im(\lambda_1(\xi))+\Re(\lambda_2(\xi)))}\right)&\\
		& \hspace{4cm} \left(\frac{\left(k-\frac{1}{2}\right)\pi+[\Im(\lambda_3(\xi))+2\Re(\lambda_4(\xi))+a(\Im(\lambda_1(\xi))+\Re(\lambda_2(\xi)))]}{\left(k-\frac{1}{2}\right)\pi+\sqrt{1+a^2}(\Im(\lambda_1(\xi))+\Re(\lambda_2(\xi)))}\right).&
	\end{eqnarray*}
\end{proposition}
\begin{proof}
	Using lemma \ref{lemma1} and (\ref{x1.1}), the Weierstrass-Enneper data for (\ref{affineSC}) is given by
	\begin{eqnarray}
	\label{2.5l}x(\xi)&=&\Im(\lambda_1(\xi))+\Re(\lambda_2(\xi))\\
	\label{2.6l}y(\xi)&=&\Im(\lambda_3(\xi))+2\Re(\lambda_4(\xi))\\
	\label{2.7l}\varphi(\xi)&=&\Re(\lambda_5(\xi)),
	\end{eqnarray}
	where $\lambda_1(\xi)=-\frac{(\iota+ a )}{\sqrt{1+a^2}}\xi$, 
	$\lambda_2(\xi)=\xi- \frac{2\sqrt{1+a^2}\tan^{-1}\left(\frac{\sqrt{a-\iota}}{\iota \sqrt{\iota+ a }}\xi\right)}{\iota\sqrt{\iota+ a }(a-\iota)^{3/2}}$,
	
	$\lambda_3(\xi)=\left(\frac{1- \iota a}{\sqrt{1+a^2}}-1\right)\xi + 2 \tan^{-1}(\xi)$,
	$\lambda_4(\xi)=\left(\frac{a\sqrt{1+a^2}\tan^{-1}\left(\frac{\sqrt{a-\iota}}{\iota\sqrt{\iota+ a }\xi}\right)}{\iota  \sqrt{\iota+ a }(a-\iota)^{3/2}}\right)$,
	
	\begin{eqnarray*}\lambda_5(\xi)&=&\frac{1}{2\sqrt{1+a^2}}(\iota+ a )\Big[2\tan^{-1}\left(a-\frac{2a}{1+\xi^2}\right)-\iota \log(a^2(\xi^2-1)^2)\\
		&&+(\xi^2+1)^2\Big]-\log(1-\xi^2).\end{eqnarray*}
	
	Let $X$ and $A$ be complex with $A$ not an odd multiple of $\frac{\pi}{2}$ (see \cite{ramanujanramanujan}), we have
	\begin{equation*}
	\frac{\cos(X+A)}{\cos(A)}=\prod_{k=1}^\infty\left[\left(1-\frac{X}{\left(k-\frac{1}{2}\pi\right)-A}\right)\left(1+\frac{X}{\left(k-\frac{1}{2}\pi\right)+A}\right)\right]
	\end{equation*}
	or
	\begin{equation}\label{2.1}
	\log\frac{\cos(X+A)}{\cos(A)}=\prod_{k=1}^\infty\log \left(\frac{\left(k-\frac{1}{2}\right)\pi-(X+A)}{\left(k-\frac{1}{2}\right)\pi-A}\right) \left(\frac{\left(k-\frac{1}{2}\right)\pi+(X+A)}{\left(k-\frac{1}{2}\right)\pi+A}\right)
	\end{equation}
	Let $X+A=y+ax$ and $A=\sqrt{1+a^2}x$ in (\ref{2.1}), then if $x$ is not an odd multiple of $\frac{\pi}{2\sqrt{1+a^2}}$, we have
	\begin{equation}\label{2.9}
	\log \frac{\cos(y+ax)}{\cos(\sqrt{1+a^2}x)}=\prod_{k=1}^\infty\log \left(\frac{\left(k-\frac{1}{2}\right)\pi-(y+ax)}{\left(k-\frac{1}{2}\right)\pi-\sqrt{1+a^2}x}\right) \left(\frac{\left(k-\frac{1}{2}\right)\pi+(y+ax)}{\left(k-\frac{1}{2}\right)\pi+\sqrt{1+a^2}x}\right)
	\end{equation}
	Substituting the W-E data in the above expression, we get
	\begin{eqnarray*}
	&&	\Re\log(\lambda_5(\xi))=\log\frac{\cos [\Im(\lambda_3(\xi))+2\Re (\lambda_4(\xi))]}{\cos[\sqrt{1+a^2}\left\{\Im(\lambda_3(\xi))+2\Re(\lambda_4(\xi))\right\}]}\\
		&&=\prod_{k=1}^\infty\log \left(\frac{\left(k-\frac{1}{2}\right)\pi-[\Im(\lambda_3(\xi))+2\Re(\lambda_4(\xi))+a(\Im(\lambda_1(\xi))+\Re(\lambda_2(\xi)))]}{\left(k-\frac{1}{2}\right)\pi-\sqrt{1+a^2}(\Im(\lambda_1(\xi))+\Re(\lambda_2(\xi)))}\right)\\
		&& \hspace{1cm}\left(\frac{\left(k-\frac{1}{2}\right)\pi+[\Im(\lambda_3(\xi))+2\Re(\lambda_4(\xi))+a(\Im(\lambda_1(\xi))+\Re(\lambda_2(\xi)))]}{\left(k-\frac{1}{2}\right)\pi+\sqrt{1+a^2}(\Im(\lambda_1(\xi))+\Re(\lambda_2(\xi)))}\right).
	\end{eqnarray*}
\end{proof}
\begin{remark}
	From (\ref{1.1}), we observe that for $a=0,$ we get $R(\omega)$ for the Scherk surface, i.e., $R(\omega)=\frac{2}{1-\omega^4}$. 
\end{remark} 
\subsection{Dirichlet series expansion of affine Scherk surface} The first Dirichlet series expansion of minimal surfaces was obtained by Dey et al. (see \cite{dey2018euler}). In this section, with the help of Ramanujan identity, we obtain the Dirichlet series expansion of affine Scherk surface. We know that
\begin{equation*}
\log \frac{\cos(y+ax)}{\cos(\sqrt{1+a^2}x)}=\prod_{k=1}^\infty \log \left(\frac{\left(k-\frac{1}{2}\right)\pi-(y+ax)}{\left(k-\frac{1}{2}\right)\pi-\sqrt{1+a^2}x}\right)
\left(\frac{\left(k-\frac{1}{2}\right)\pi+(y+ax)}{\left(k-\frac{1}{2}\right)\pi+\sqrt{1+a^2}x}\right)
\end{equation*}
Set $\left(k-\frac{1}{2}\right)\pi=\alpha_k$, we can write the above expression as
\begin{eqnarray*}
	\log \frac{\cos(y+ax)}{\cos(\sqrt{1+a^2}x)}&=&\prod_{k=1}^\infty \log \left(\frac{\alpha_k-(y+ax)}{\alpha_k-\sqrt{1+a^2}x}\right)
	\left(\frac{\alpha_k+(y+ax)}{\alpha_k+\sqrt{1+a^2}x}\right)\\
	&=&\prod_{k=1}^\infty \log \left(1+\frac{(1+a^2)x^2}{\alpha_k^2-(1+a^2)x^2}\right) \left(1-\frac{(y+ax)^2}{\alpha_k^2}\right)
\end{eqnarray*}
or
\begin{eqnarray}\label{a2.10}
\nonumber &&\log \frac{\cos(y+ax)}{\cos(\sqrt{1+a^2}x)}= \sum_{k=1}^\infty\Big[\frac{(1+a^2)x^2}{\alpha_k^2-(1+a^2)x^2}-\frac{1}{2}\left(\frac{(1+a^2)x^2}{\alpha_k^2-(1+a^2)x^2}\right)^2 \\\nonumber &&+\frac{1}{3}\left(\frac{(1+a^2)x^2}{\alpha_k^2-(1+a^2)x^2}\right)^3- \cdots \Big]-\sum_{k=1}^\infty\Big[\frac{(y+ax)^2}{\alpha_k^2}+\frac{1}{2}\left(\frac{(y+ax)^2}{\alpha_k^2}\right)^2\\
&&+\frac{1}{3}\left(\frac{(y+ax)^2}{\alpha_k^2}\right)^3+\cdots\Big],
\end{eqnarray}
where $|\frac{(1+a^2)x^2}{\alpha_k^2-(1+a^2)x^2}|<1$ and $|\frac{(y+ax)^2}{\alpha_k^2}|<1,$ $\forall$ integer $k$.
This implies that
\begin{equation}\label{3.1}
\log \frac{\cos(y+ax)}{\cos(\sqrt{1+a^2}x)}=\sum_{k=1}^\infty\left[P_k\left(1,\frac{(1+a^2)x^2}{\alpha_k^2-(1+a^2)x^2}\right)-T_k\left(1,\frac{(y+ax)^2}{\alpha_k^2}\right)\right],
\end{equation}
where $P_k(s,a)=\sum_{n=1}^\infty\left[(-1)^{n+1}\frac{a^n}{n^s}\right]$,  $T_k(s,b)=\sum_{n=1}^\infty \frac{b^n}{n^s}$ are two Dirichlet series with real parameters $a,$ $b$, respectively. Since $\xi$, $x, y$ are related by $(\ref{2.5l})-(\ref{2.6l}$ and $|\frac{(1+a^2)x^2}{\alpha_k^2-(1+a^2)x^2}|<1$ and $|\frac{(y+ax)^2}{\alpha_k^2}|<1,$ $\forall$ integer $k$. Therefore for a cosiderably small bounded domain of $\xi$ and on substituting the W-E data in (\ref{3.1}), we can state the following:
\begin{proposition}
	For a small enough domain of $\xi,$ the Dirichlet series expansion of the non-trivial affine minimal translation surface is given by
	\begin{eqnarray*}
		&	\log \frac{\cos(y+ax)}{\cos(\sqrt{1+a^2}x)}=\Re\log(\lambda_5(\xi))=\sum_{k=1}^\infty\Big[P_k\left(1,\frac{(1+a^2)[\Im(\lambda_1(\xi))+\Re(\lambda_2(\xi))]^2}{\alpha_k^2-(1+a^2)[\Im(\lambda_1(\xi))+\Re(\lambda_2(\xi))]^2}\right)&\\
		&\hspace{4cm}-T_k\left(1,\frac{[\Im(\lambda_3(\xi))+2\Re(\lambda_4(\xi))+a(\Im(\lambda_1(\xi))+\Re(\lambda_2(\xi)))]^2}{\alpha_k^2}\right)\Big].&
	\end{eqnarray*}
\end{proposition}
\subsection{Logrithmic distribution of affine Scherk surface}
It is easy to show the convergence of both the Dirichlet series $P_k$ and $T_k$ of (\ref{3.1}), which in turn implies the convergence of both the series in RHS of (\ref{a2.10}). Now we re-examine the RHS of (\ref{a2.10}) in the way: 
\begin{equation}\label{h2.12}
\log \frac{\cos(y+ax)}{\cos(\sqrt{1+a^2}x)}= \sum_{k=1}^\infty\left(-A_k-\frac{1}{2}A_k^2-\frac{1}{3}A_k^3-\cdots\right),
\end{equation}
where $A_k^i=\left\{\frac{(y+ax)^2}{\alpha_k^2}\right\}^i+\left\{\frac{-(1+a^2)x^2}{\alpha_k^2-(1+a^2)x^2}\right\}^i$ and $0<A_k<1.$
The RHS of the above converges for each $k$ as it the sum of two convergent series. Supposing that the series in (\ref{h2.12}) converges to a constant $s$. From the RHS of (\ref{h2.12}), we have
\begin{equation*}
\log(1-A_k)=\sum_{j=1}^\infty -\frac{A_k^j}{j}.
\end{equation*}
Since the above identity holds for each $k$, we can write
\begin{equation}\label{j2.13}
\sum_{k=1}^n\log(1-A_k)=\sum_{k=1}^n\sum_{j=1}^\infty -\frac{A_k^j}{j},
\end{equation}
where $$\sum_{k=1}^n\sum_{j=1}^\infty -\frac{A_k^j}{j}=s_n$$ are the partial sums whose limiting sum tends to $s$, i,e.,
\begin{equation}\label{a2.12}
s=\log \frac{\cos(y+ax)}{\cos(\sqrt{1+a^2}x)}= \sum_{k=1}^\infty\sum_{j=1}^\infty\left(s_1^j +s_2^j+\cdots+s_k^j+\cdots\right).
\end{equation}
Now we can write (\ref{j2.13}) as
\begin{equation}\label{ht2.14}
\sum_{j=1}^\infty \left(\sum_{k=1}^n \frac{-A_k^j}{j}\right)=\sum_{k=1}^{n}\log(1-A_k).
\end{equation}
From the LHS of above expression, we get an indexed set of real values $\{s_n\}_{n\in \mathbb{N}}$ whose total limiting sum is $s.$
The expression in (\ref{ht2.14}) can be combined in a way such that:
\begin{equation*}\label{t2.14}
\sum_{j=1}^\infty\left(\frac{\sum_{k=1}^n\left( \frac{-A_k^j}{j}\right)}{\sum_{k=1}^{n}\log(1-A_k)}\right)=1.
\end{equation*}
This gives us the probability mass function ($f_k(j)$) of a $\log(A_k)$ distributed random variable which assumes the real values of the partial sums, i.e., $f_k(j)=\frac{\sum_{k=1}^n\left( \frac{-A_k^j}{j}\right)}{\sum_{k=1}^{n}\log(1-A_k)},$ $j \geq1$ with the parameter $0< A_k<1.	$  
Looking back at (\ref{h2.12}), we see that the values of the $\log(A_k)$ distributed random variable are the values assumed by the partial sums with the limiting sum equal to the height function of the affine Scherk surface. Thus we can state the following:
\begin{proposition} Let $\log \frac{\cos(y+ax)}{\cos(\sqrt{1+a^2}x)}=s$ be the height function of the affine Scherk surface, such that $s_n$ are the convergent partial sum series with the limiting sum equal to the height function of affine Scherk surface.
	Then the probability mass function of the random variable assuming the values $s_n$ is given by $f_k(j)=\frac{\sum_{k=1}^n\left( \frac{-A_k^j}{j}\right)}{\sum_{k=1}^{n}\log(1-A_k)}$, $j \geq 1,$ with the parameter $0<A_k<1.	$
\end{proposition}
The probability mass function gives the probability that a discrete random variable is exactly equal to some value (see \cite{khuri2003advanced}).

\subsection{Wick rotation of the affine minimal translation surface:}\label{ob2.3} By a wick rotation in the variable $y \rightarrow \iota y$ of theorem $\ref{thm1.1}$, we get a Born-Infeld soliton analogue. We shall call a wick rotated affine minimal translation surface by affine wick minimal translation surface (AWMT). Thus from theorem \ref{thm1.1}, one can observe:

{\bf Observation:} Let $\varphi$ be an AWMT surface. Then $\varphi$ is either a linear function or is given by
$\varphi(x,y)=\log \frac{\cos (\iota y+ax)}{\cos(\sqrt{1+a^2}x)}.$

\section{Affine Scherk surface and Ramanujan identity in $\mathbb{L}^3$}\label{section3}
In \cite{dey2017born}, Dey and Singh obtained the Weierstrass data and the corresponding Ramanujan identity expression of Scherk surface in $\mathbb{L}^3$. In this section, we generalize the notion to a more generalized class of Scherk surfaces, i.e., affine Scherk surfaces.
\begin{lemma}
	Let $\varphi$ be a maximal affine translation surface. Then $\varphi$ is either a linear function or up to dilation and translation is given by
	\begin{equation}\label{affineSC1}
	\varphi(x,y)=\log \left[\frac{\cosh (y+ax)}{\cosh(\sqrt{1+a^2}x)}\right].
	\end{equation}
\end{lemma}
\begin{proof}
	Let $\varphi(x,y)=f(x)+g(y+ax)$ be an affine translation surface in $\mathbb{L}^3$ with the metric $g=dx_1^2+dx_2^2-dx_3^2.$ The first and the second fundamental form coefficients are:
	$E=1-(f^\prime +a g^\prime)^2$, $F=-g^\prime(f^\prime +a g^\prime)$, $G=1-g^\prime$ and $L=-(f^{\prime \prime}+a^2 g^{\prime \prime})D^{-1},$  $M=-ag^{\prime\prime}D^{-1},$ $N=-g^{\prime \prime}D^{-1}$, where $f^\prime=\frac{df(x)}{dx}$, $g^\prime=\frac{dg(y+ax)}{d(y+ax)}$ and $D^2=1-{g^\prime}^2-(f^\prime+ag^\prime)^2$.  Therefore, the maximal surface equation is given by 
	\begin{equation}\label{2.11}
	f^{\prime \prime}(1-{g^\prime}^2)+g^{\prime \prime}(1+a^2-{f^\prime}^2)=0.\end{equation}
	By a direct computation of (\ref{2.11}), we get the expression (\ref{affineSC1}).
\end{proof}
We call the surface in (\ref{affineSC1}) as the affine Scherk surface of first kind.
\begin{lemma}\label{lemma2} $R(\omega)$ for the affine Scherk surface of the first kind is given by
	\begin{equation}\label{q1.1}
	R(\omega)= \frac{\sqrt{1+a^2}(1+\iota a)}{(1+a^2)+(1-\iota a)^2 \omega^2}+ \frac{1}{1-\omega^2},
	\end{equation}
	where $a\in \mathbb{R}.$
	\begin{proof}
		Introducing the new variables $\zeta=\frac{1-\sqrt{1-4uv}}{2v},$  $\bar{\zeta}=\frac{1-\sqrt{1-4uv}}{2u}$	with the inverse transformations $u=\frac{\zeta}{1+\zeta \bar{\zeta}}$, $v=\frac{\bar{\zeta}}{1+\zeta \bar{\zeta}}$. Following the similar lines as in lemma \ref{lemma1}, we get the claim.
	\end{proof}
\end{lemma}
\begin{remark}The points $\pm 1, \pm \iota \frac{\sqrt{1+a^2}}{1-\iota a}$ act as umbilical points. So, the parameterization is to be considered away from these possible points.
\end{remark}
\begin{proposition}\label{prop4} Let $\varphi$ be an affine non-planar minimal translation surface in $\mathbb{L}^3$ and $\xi$ in a neighbourhood of $\mathbb{C}$ away from the umbilical points $(\xi \neq\pm 1, \pm \iota \frac{\sqrt{1+a^2}}{1-\iota a})$, we have the following identity
	\begin{eqnarray*}
		&	\log \frac{\cosh(y+ax)}{\cosh(\sqrt{1+a^2}x)}=\prod_{k=1}^\infty\log \left(\frac{\left(k-\frac{1}{2}\right)\pi-\iota[\Re(\mu_3(\zeta))+\Im(\mu_4(\zeta))+a\{\Re(\mu_1(\zeta))+\Im(\mu_2(\zeta))\}]}{\left(k-\frac{1}{2}\right)\pi-\iota\sqrt{1+a^2}(\Re(\mu_1(\zeta))+\Im(\mu_2(\zeta)))}\right)&\\
		& \hspace{4cm} \left(\frac{\left(k-\frac{1}{2}\right)\pi+\iota[\Re(\mu_3(\zeta))+\Im(\mu_4(\zeta))+a\{\Re(\mu_1(\zeta))+\Im(\mu_2(\zeta))\}]}{\left(k-\frac{1}{2}\right)\pi+\iota\sqrt{1+a^2}\{\Re(\mu_1(\zeta))+\Im(\mu_2(\zeta))\}}\right).&
	\end{eqnarray*}
\end{proposition}
\begin{proof}
	Using lemma \ref{lemma2} and (\ref{x1.1}), the Weierstrass-Enneper data for (\ref{affineSC}) is given by
	\begin{eqnarray*}
		x(\zeta)&=&\Re(\mu_1(\zeta))+\Im(\mu_2(\zeta))\\
		y(\zeta)&=&\Re(\mu_3(\zeta))+\Im(\mu_4(\zeta))\\
		\varphi(\zeta)&=&\Re(\mu_5(\zeta))+\Re(\mu_6(\zeta)),
	\end{eqnarray*}
	where 
	$\mu_1=-\frac{\left(a^2+\iota \left(\sqrt{a^2+1}+2\right) a+\sqrt{a^2+1}-1\right) \zeta}{(\iota+ a )^2}+\log \frac{1+\zeta}{1-\zeta},$
	
	$\mu_2=-2 \frac{a\sqrt{\iota -a }  \sqrt{a^2+1}}{(\iota+ a )^{5/2}} \tan ^{-1}\left(\frac{\sqrt{\iota+ a } }{\sqrt{\iota -a }}\zeta\right)$,
	
	$\mu_3=\frac{\sqrt{\iota -a } \sqrt{a^2+1}}{(\iota+ a )^{5/2}} \left(\log \left(\frac{\sqrt{\iota -a }-\iota \sqrt{\iota+ a } \zeta}{\sqrt{\iota -a }+\iota \sqrt{\iota+ a } \zeta}\right)\right),$
	
	$\mu_4=-\frac{\left(a^2+\iota \left(\sqrt{a^2+1}+2\right) a+\sqrt{a^2+1}-1\right) \zeta}{(\iota+ a )^2}$,
	$\mu_5=\log \frac{1+\zeta}{1-\zeta}$,
	
	$\mu_6=-2 \frac{\sqrt{\iota -a } \sqrt{a^2+1}}{(\iota+ a )^{3/2}} \tan ^{-1}\left(\frac{\sqrt{\iota+ a } }{\sqrt{\iota -a }}\zeta\right).$
	
	From the Ramanujan's identity, we have
	\begin{equation}\label{3.7}
	\log\frac{\cos(X+A)}{\cos(A)}=\prod_{k=1}^\infty\log \left(\frac{\left(k-\frac{1}{2}\right)\pi-(X+A)}{\left(k-\frac{1}{2}\right)\pi-A}\right) \left(\frac{\left(k-\frac{1}{2}\right)\pi+(X+A)}{\left(k-\frac{1}{2}\right)\pi+A}\right).
	\end{equation}
	Substituting $X+A=\iota (y+ax)$ and $A=\iota \sqrt{1+a^2}x$ in (\ref{3.7}), where $\iota \sqrt{1+a^2}x$ is not an odd multiple of $\frac{\iota \pi}{2}$, we obtain
	\begin{equation*}
	\log \frac{\cosh(y+ax)}{\cosh(\sqrt{1+a^2}x)}=\prod_{k=1}^\infty\log \left(\frac{\left(k-\frac{1}{2}\right)\pi-\iota(y+ax)}{\left(k-\frac{1}{2}\right)\pi-\iota\sqrt{1+a^2}x}\right)
	\left(\frac{\left(k-\frac{1}{2}\right)\pi+\iota(y+ax)}{\left(k-\frac{1}{2}\right)\pi+\iota\sqrt{1+a^2}x}\right).
	\end{equation*}
	Substituting the W-E data in the above identity, we get the claim.
\end{proof}
\begin{remark} In case of affine Scherk surfaces in $\mathbb{L}^3$, we have now its W-E data and proposition $\ref{prop4}$, so on the similar lines as in section $\ref{sec2}$, we can easily find the analogous forms of Born-Infeld soliton, Dirichlet series expansion and the probability density function. 
\end{remark}

\section*{Acknowledgement} I am very thankful to Prof. Rukmini Dey for having valuable discussions and suggestions while preparing this article.
\bibliography{references} 
\bibliographystyle{alpha}
\end{document}